\documentclass[reqno,11pt]{amsart}
\usepackage[foot]{amsaddr}
\usepackage{bbold}
\usepackage{charter}
\usepackage{enumitem}
\usepackage[margin=1in]{geometry}
\usepackage[colorlinks=true,linktoc=all,linkcolor=blue,citecolor=blue]{hyperref}

\theoremstyle{plain}
\newtheorem{theorem}{Theorem}[section]
\newtheorem*{theorem*}{Theorem}
\newtheorem{proposition}[theorem]{Proposition}
\newtheorem{lemma}[theorem]{Lemma}
\newtheorem{corollary}[theorem]{Corollary}
\theoremstyle{definition}

\newtheorem*{definition*}{Definition}
\newtheorem{remark}[theorem]{Remark}

\newtheorem{conjecture}[theorem]{Conjecture}
\numberwithin{equation}{section}
\everymath{\displaystyle}
\allowdisplaybreaks

\newcommand{\B}{\mathbb{B}}
\newcommand{\C}{\mathbb{C}}
\newcommand{\D}{\mathbb{D}}
\newcommand{\N}{\mathbb{N}}

\newcommand{\Bloch}{\mathcal{B}}
\newcommand{\lilBloch}{\mathcal{B}_0}
\newcommand{\norm}[1]{\left|\left|#1\right|\right|}
\newcommand{\supnorm}[1]{\norm{#1}_\infty}
\newcommand{\blochnorm}[1]{\norm{#1}_\Bloch}
\renewcommand{\mod}[1]{\left|#1\right|}
\newcommand{\conj}[1]{\overline{#1}}
\newcommand{\Aut}{\mathrm{Aut}}
\newcommand{\distbd}[1]{\partial^*\hspace{-.4ex} #1}

\title[Multiplication Operators on $\Bloch(D)$]{Multiplication Operators on the Bloch Space\\ of Bounded Homogeneous Domains}

\author{Robert F.~Allen and Flavia Colonna}

\address{Department of Mathematical Sciences, George Mason University}
\email{rallen2@gmu.edu, fcolonna@gmu.edu}

\date{}


\begin{document}

%-------------------------------------------------------------------------------------------------
% Abstract
%-------------------------------------------------------------------------------------------------
\begin{abstract} In this paper, we study the multiplication operators on the Bloch space of a bounded homogeneous domain in $\C^n$. Specifically, we characterize the bounded and the compact multiplication operators, establish estimates on the operator norm, and determine the spectrum. We prove that the only bounded multiplication operators on the Bloch space of the polydisk are those whose symbol is constant. Furthermore, we prove that for a large class of bounded symmetric domains, the isometric multiplication operators are those whose symbol is a constant of modulus one.
\end{abstract}

\maketitle

%-------------------------------------------------------------------------------------------------
% Introduction
%-------------------------------------------------------------------------------------------------
\section{Introduction}
In recent years, the operator theory of many functional Banach spaces that arise in complex function theory has been studied extensively.  Two particular important classes of operators are the multiplication operators and the composition operators, defined as $M_\psi f = \psi f$ and $C_\varphi f = f\circ\varphi$, respectively.  While a great deal is known about the composition operators on a wide variety of functional Banach spaces (see \cite{Shapiro:93} and \cite{CowenMacCluer:95}), in the literature the amount on the study of multiplication operators on some spaces has been surprisingly small.

Multiplication operators on the Bloch space of the open unit disk $\mathbb D$ have been studied in \cite{Arazy:82}, \cite{BrownShields:91}, \cite{OhnoZhao:01}, and \cite{AllenColonna:08}.  In this setting, criteria for the boundedness of multiplications operators were obtained independently by Arazy \cite{Arazy:82}, and by Brown and Shields \cite{BrownShields:91}.  More recently, Ohno and Zhao characterized the bounded and the compact weighted composition operators on the Bloch space and on its subspace given by the closure of the polynomials known as the {\it little Bloch space} \cite{OhnoZhao:01}.  As a corollary, they deduced that the only compact multiplication operators are those whose symbol is identically zero.  In \cite{AllenColonna:08}, the authors gave estimates on the operator norm, determined the spectrum, and showed that the only isometric multiplication operators are those whose symbol is a unimodular constant.

Bloch functions in higher dimensions have been introduced on bounded homogeneous domains in \cite{Hahn:75}. % and \cite{Timoney:80-I}.
Krantz and Ma defined the notion of Bloch function on a strongly pseudoconvex domain  \cite{KrantzMa:88}.

The study of the multiplication operators on the Bloch space in higher dimensions was begun by Zhu, who characterized the bounded multiplication operators on the Bloch space and the little Bloch space of the unit ball $\B_n$ \cite{Zhu:04}.  In \cite{ZhouChen:05} and \cite{ZhouChen-I:05}, Zhou and Chen characterized the bounded and the compact weighted composition operators on the Bloch space of the unit ball and the unit polydisk, thereby obtaining corresponding results for the multiplication operators.

In this work, we carry out the study of the multiplication operators on the Bloch space of a bounded homogeneous domain and on a subspace we call the $*$-little Bloch space.  Specifically, we extend the results of Zhu, Zhao and Chen, and obtain operator norm estimates similar to those found for the case of the unit disk.  Furthermore, we determine the spectrum and show that for the the polydisk and for bounded symmetric domains that do not have the disk as a factor, the only isometric multiplication operators are those whose symbol is a constant function of modulus one.

%-------------------------------------------------------------------------------------------------
% Preliminaries
%-------------------------------------------------------------------------------------------------
\section{Preliminaries}\label{section:preliminaries}
A domain $D \subset \C^n$ is \emph{homogeneous} if the group $\Aut(D)$ of biholomorphic self-maps of $D$ (which we call \textit{automorphisms}) acts transitively on $D$.  An important subclass of homogeneous domains are the symmetric domains.  A domain $D \subset \C^n$ is \emph{symmetric at $z_0 \in D$} if there exists an involutive automorphism of $D$ for which $z_0$ is an isolated fixed point.  A domain $D$ is said to be \emph{symmetric} if $D$ is symmetric at each point.

Cartan \cite{Cartan:35} showed that every bounded symmetric domain in $\C^n$ is biholomorphic to a finite product of irreducible bounded symmetric domains, unique up to order.  Moreover, Cartan classified the irreducible bounded symmetric domains into six classes; four of the classes are referred to as \emph{Cartan classical domains}, whereas the other two, each consisting of a single domain, are referred to as the \emph{exceptional domains}.  A bounded symmetric domain written as such a product is said to be in \emph{standard form}.  The Cartan classical domains are defined as:
$$\begin{aligned}
R_I &= \{Z \in \mathcal{M}_{m,n}(\C) : I_m - ZZ^* > 0\}, \text{ for $m \geq n \geq 1$},\\
R_{II} &= \{Z \in \mathcal{M}_n(\C) : Z=Z^T, I_n - ZZ^* > 0\}, \text{ for $n \geq 1$},\\
R_{III} &= \{Z \in \mathcal{M}_n(\C) : Z = -Z^T, I_n - ZZ^* > 0\}, \text{ for $n \geq 2$},\\
R_{IV} &= \left\{z=(z_1,\dots,z_n) \in \C^n : A > 0, \norm{z}^2 < 1\right\}, \text{ for } 1 \leq n \neq 2,
\end{aligned}$$
where $\mathcal{M}_{m,n}(\C)$ denotes the set of $m\times n$ matrices with entries in $\C$, $\mathcal{M}_n(\C) = \mathcal{M}_{n,n}(\C)$, $Z^T$ is the transpose of $Z$, $A = \mod{\sum z_j^2}^2 + 1 - 2\norm{z}^2$, and $\norm{z}^2 = \sum \mod{z_j}^2$.  To assure these classes are disjoint, the dimensional restrictions $n \geq 2$ for domains in $R_{II}$ and $n \geq 5$ for domains in $R_{III}$ and $R_{IV}$ are imposed.  For a description of the exceptional domains we refer the reader to \cite{Drucker:78}.

\begin{remark} Every bounded symmetric domain is homogeneous and a bounded homogeneous domain that is symmetric at a single point is symmetric.  Each bounded homogeneous domain $D$ is endowed with a canonical metric invariant under the action of $\Aut(D)$, called the \textit{Bergman metric}.  The unit ball $\B_n$ and the unit polydisk $\D^n$ are examples of bounded symmetric domains in $\C^n$.  A description of the Bergman metric and the automorphism group of the Cartan classical domains can be found in \cite{Kobayashi:05}.  Cartan proved that every bounded homogeneous domain in dimensions 2 and 3 is symmetric.  An example of a homogeneous domain in dimension 4 that is not symmetric can be found in \cite{Pjat:59}.
\end{remark}

In \cite{Hahn:75} Hahn introduced the notion of Bloch function on a bounded homogeneous domain.  In \cite{Timoney:80-I} and \cite{Timoney:80-II} Timoney expanded the study of  Bloch functions in this setting.  We denote by $H(D)$ the set of holomorphic functions from $D \subset \C^n$ into $\C$ and by $H^\infty(D)$ the set of bounded functions in $H(D)$.

Let $D$ be a bounded homogeneous domain in $\C^n$ and $f \in H(D)$.  Then for $z \in D$, define \begin{equation}\label{equation:Q_f}Q_f(z) = \sup_{u \in \C^n\setminus\{0\}} \frac{\mod{\nabla( f)(z)u}}{H_z(u,\conj{u})^{1/2}}\end{equation} where for $u = (u_1,\dots,u_n) \in \C^n$, $\nabla(f)(z)u = \displaystyle\sum_{j=1}^n \frac{\partial f}{\partial z_j}(z)u_j$, and $H_z(\cdot,\cdot)$ is the Bergman metric of $D$. A function $f$ is called \emph{Bloch} if $$\beta_f = \sup_{z \in D}\;Q_f(z) < \infty.$$  It was shown by Timoney that every bounded holomorphic function on $D$ is Bloch (see \cite{Timoney:80-I}, Example 3.7(2)).  For a fixed $z_0 \in D$, the \emph{Bloch space} on $D$ based at $z_0$ is the Banach space $\Bloch(D)$ of all Bloch functions on $D$ under the norm $$\blochnorm{f} = \mod{f(z_0)} + \beta_f.$$  For convenience, we shall assume that $D$ contains the origin and choose 0 as the base point to define the norm.  Note that if $D = \D$, $Q_f(z) = (1-\mod{z}^2)\mod{f'(z)}.$

In \cite{Timoney:80-II} the little Bloch space on the unit ball was defined as
$$\lilBloch(\B_n) = \left\{f \in \Bloch(\B_n) : \lim_{\norm{z} \to 1^-} Q_f(z) = 0\right\}$$which is the closure of the polynomials in $\Bloch(\B_n)$. However, when $D$ is a bounded symmetric domain other than the ball, the only Bloch functions $f$ on $D$ such that $Q_f(z)\to 0$ as $z$ approaches the boundary $\partial D$ of $D$ are the constants, so the little Bloch space $\lilBloch(D)$ is defined as the closure of the polynomials.

The $*$-little Bloch space is defined as $$\Bloch_{0^*}(D) = \left\{f \in H(D) : \lim_{z \to \distbd{D}}\;Q_f(z) = 0\right\},$$ where $\distbd{D}$ denotes the distinguished boundary of $D$.  %If $D$ is a bounded symmetric domain in $\C^n$ other than $\B_n$, then
%$$\left\{f \in \Bloch(D) : \lim_{z \to \partial D} Q_f(z) = 0\right\}$$ is the set of constant functions, which is why $\Bloch_{0^*}(D)$ is defined in terms of the distinguished boundary.
If $D$ is the unit ball, then $\partial D = \distbd{D}$ and thus $\lilBloch(D) = \Bloch_{0^*}(D)$, while when $D\neq \B_n$, $\lilBloch(D)$ is a proper subspace of $\Bloch_{0^*}(D)$, and $\Bloch_{0^*}(D)$ is a non-separable subspace of $\Bloch(D)$.

%-------------------------------------------------------------------------------------------------
% Boundedness
%-------------------------------------------------------------------------------------------------
\section{Boundedness}
Let $D$ be a bounded homogeneous domain in $\C^n$.  For $z \in D$, we define
$$\begin{aligned}
\omega(z) &= \sup_{f \in \Bloch(D)}\left\{\mod{f(z)} : f(0) = 0 \text{ and }\blochnorm{f} \leq 1\right\},\\
\omega_0(z) &= \sup_{f \in \Bloch_{0^*}(D)}\left\{\mod{f(z)} : f(0) = 0 \text{ and }\blochnorm{f} \leq 1\right\}.
\end{aligned}$$  In \cite{AllenColonna:07} it was shown that if $f \in \Bloch(D)$, then $$\beta_f = \sup_{z \neq w}\;\frac{\mod{f(z) - f(w)}}{\rho(z,w)},$$ where $\rho$ is the distance induced by the Bergman metric on $D$.  Consequently, if $f \in \Bloch(D)$ with $f(0) = 0$ and $\blochnorm{f} \leq 1$, then $\mod{f(z)} \leq \rho(z,0)$ for each $z\in D$.  Thus  \begin{equation}\label{omega_0_omega_rho}\omega_0(z) \leq \omega(z) \leq \rho(z,0),\end{equation} and hence $\omega_0(z)$ and $\omega(z)$ are finite for all $z \in D$.

\begin{lemma}\label{lemma:mod_f_inequality} Let $D$ be a bounded homogeneous domain in $\C^n$ and $z \in D$.
\begin{enumerate}
\item[(a)] If $f \in \Bloch(D)$, then $$\mod{f(z)} \leq \mod{f(0)} + \omega(z)\beta_f.$$
\item[(b)] If $f \in \Bloch_{0^*}(D)$, then $$\mod{f(z)} \leq \mod{f(0)} + \omega_0(z)\beta_f.$$
\end{enumerate}
\end{lemma}

\begin{proof}  Suppose $f \in \Bloch(D)$ is not constant.  Then the function $g$ defined by $g(z) = \frac{1}{\beta_f}(f(z)-f(0))$ is holomorphic and such that $g(0) = 0$ and $\blochnorm{g} = 1$.  So $\mod{g(z)} \leq \omega(z)$ for all $z \in D$.  Thus, $$\mod{f(z)} \leq \mod{f(0)} + \mod{f(z)-f(0)} = \mod{f(0)} + \mod{g(z)}\beta_f \leq \mod{f(0)} + \omega(z)\beta_f,$$ for all $z \in D$. The result for $\Bloch_{0^*}(D)$ is analogous.\end{proof}

For $\psi \in H(D)$, define
$$\begin{aligned}
\sigma_\psi &= \sup_{z \in D}\;\omega(z)Q_\psi(z),\\
\sigma_{0,\psi} &= \sup_{z \in D}\;\omega_0(z)Q_\psi(z).
\end{aligned}$$

\begin{theorem}\label{theorem:bounded operator on bloch space} Let $D$ be a bounded homogeneous domain in $\C^n$ and $\psi \in H(D)$.
\begin{enumerate}
\item[(a)] $M_\psi$ is bounded on $\Bloch(D)$ if and only if $\psi \in H^\infty(D)$ and $\sigma_\psi < \infty$.
\item[(b)] $M_\psi$ is bounded on $\Bloch_{0^*}(D)$ if and only if $\psi \in H^\infty(D) \cap \Bloch_{0^*}(D)$ and $\sigma_{0,\psi} < \infty$.
\end{enumerate}
\end{theorem}

\begin{proof}  To prove (a), assume $\psi \in H^\infty(D)$ with $\sigma_\psi < \infty$, and let $f \in \Bloch(D)$.  Applying the product rule, for all $z \in D$ we have
\begin{equation}\label{equation:blochnorm inequality 1}Q_{\psi f}(z) \leq \mod{\psi(z)}Q_f(z) + \mod{f(z)}Q_\psi(z).\end{equation}  Since $\psi$ is bounded, using Lemma \ref{lemma:mod_f_inequality}(a), inequality (\ref{equation:blochnorm inequality 1}) yields $$\begin{aligned}\label{equation:operator inequality}\beta_{\psi f} &\leq \supnorm{\psi}\beta_f + \sup_{z \in D}\;(\mod{f(0)} + \omega(z)\beta_f)Q_\psi(z)\\ &\leq \supnorm{\psi}\beta_f + \mod{f(0)}\beta_\psi + \sigma_\psi\beta_f,\end{aligned}$$ which is finite.  Thus, $M_\psi f \in \Bloch(D)$ and \begin{equation}\label{equation:blochnorm_inequality_2}\blochnorm{M_\psi f} \leq \mod{f(0)}\blochnorm{\psi} + (\supnorm{\psi} + \sigma_\psi)\beta_f.\end{equation}  It follows immediately that $\blochnorm{M_\psi f} \leq \left(\blochnorm{\psi} + \supnorm{\psi} + \sigma_\psi\right)\blochnorm{f},$  proving that $M_\psi$ is bounded on $\Bloch(D)$.

Conversely, suppose $M_\psi$ is bounded on $\Bloch(D)$.  By Lemma 11 of \cite{DurenRombergShields:69}, $\psi \in H^\infty(D)$ with $\supnorm{\psi} \leq \norm{M_\psi}$.  Thus, it suffices to show that $\sigma_\psi$ is finite.  Let $f \in \Bloch(D)$.  Applying the product rule, for each $z \in D$ we obtain $$\begin{aligned}\mod{f(z)}Q_\psi(z) &\leq Q_{\psi f}(z) + \mod{\psi(z)}Q_f(z) \leq \blochnorm{M_\psi f} + \mod{\psi(z)}Q_f(z)\\
&\leq \left(\norm{M_\psi} + \mod{\psi(z)}\right)\blochnorm{f}.\end{aligned}$$  Taking the supremum over all Bloch functions $f$ such that $f(0) = 0$ and $\blochnorm{f} \leq 1$, we get
$\omega(z)Q_\psi(z) \leq \norm{M_\psi} + \mod{\psi(z)}.$  Finally, taking the supremum over all $z \in D$, we arrive at $\sigma_\psi \leq \norm{M_\psi} + \supnorm{\psi},$ which is finite.

Next, assume $\psi \in H^\infty(D) \cap \Bloch_{0^*}(D)$ and $\sigma_{0,\phi} < \infty$.  Then for $f \in \Bloch_{0^*}(D)$ and $z \in D$, we have
$Q_{\psi f}(z) \leq \supnorm{\psi}Q_f(z) + (\mod{f(0)} + \omega_0(z)\beta_f)Q_\psi (z) \to 0$ as $z \to \distbd{D}$.  The boundedness of $M_\psi$ on $\Bloch_{0^*}(D)$ follows by using the argument in the proof of part (a) and Lemma \ref{lemma:mod_f_inequality}(b).

Conversely, if $M_\psi$ is bounded on $\Bloch_{0^*}(D)$, then $\psi = M_\psi(1) \in \Bloch_{0^*}(D)$ since $1 \in \Bloch_{0^*}(D)$.  Arguing as in the proof of part (a), we see that $\sigma_{0,\psi} \leq \norm{M_\psi} + \supnorm{\psi},$ completing the proof.
\end{proof}

\subsection{Equivalence of Boundedness on $\Bloch(D)$ and $\Bloch_{0^*}(D)$}
We will now discuss conditions for which the boundedness of $M_\psi$ is equivalent on the Bloch space and $*$-little Bloch space.  We first prove that for the cases of the unit ball and polydisk, the boundedness of $M_\psi$ on $\Bloch(D)$ is equivalent to the boundedness on $\Bloch_{0^*}(D)$.  We then show sufficient conditions for which this is true on a general bounded homogeneous domain.

In the case of the unit ball, the quantities $\omega(z)$ and $\omega_0(z)$ are equal to the Bergman distance from 0 to $z$ for all $z \in \B_n$, and we have the explicit formula \begin{equation}\label{equation:omega_ball}\omega_0(z) = \omega(z) = \frac{1}{2}\log\frac{1+\norm{z}}{1-\norm{z}}\end{equation} (see \cite{Zhu:04}, Theorems 3.9 and 3.14).  In turn, this implies that $\sigma_{0,\psi} = \sigma_\psi$.  So, we have the following characterization of bounded multiplication operators on the Bloch space and little Bloch space of the unit ball.  Recall that for the unit ball the little Bloch space and the $*$-little Bloch space are the same.

\begin{corollary}\label{corollary:bounded_on_ball} Let $\psi \in H(\B_n)$.  Then the following are equivalent:
\begin{enumerate}
\item[(a)] $M_\psi$ is bounded on $\Bloch(\B_n)$.
\item[(b)] $M_\psi$ is bounded on $\lilBloch(\B_n)$.
\item[(c)] $\psi \in H^\infty(\B_n)$ and $\displaystyle\sup_{z \in \B_n}\;\log\frac{1+\norm{z}}{1-\norm{z}}Q_\psi(z)$ is finite.
\end{enumerate}
\end{corollary}

\begin{proof} The equivalence $(a) \Longleftrightarrow (c)$ and the implication $(b) \Longrightarrow (c)$ follow immediately from Theorem \ref{theorem:bounded operator on bloch space} and (\ref{equation:omega_ball}).  To show that $(c) \Longrightarrow (b)$, it suffices to show that $\psi \in \lilBloch(\B_n)$.  Since $\displaystyle\sup_{z \in \B_n}\log\frac{1+\norm{z}}{1-\norm{z}}Q_\psi(z)$ is finite, and as $\norm{z} \to 1^-$, $\log\frac{1+\norm{z}}{1-\norm{z}}$ goes to $\infty$, it must be the case that $Q_\psi(z) \to 0$.  Thus $\psi \in \lilBloch(\B_n)$, as desired.
\end{proof}

We will now show the analogous result for the Bloch space on the polydisk.  We first need the following lemma.

\begin{lemma}\label{lemma_polydisk} For $z \in \D^n$ and $k = 1,\dots,n$, the following inequalities hold:
\begin{enumerate}
\item[(a)] $\displaystyle\frac{1}{2}\log\frac{1+\mod{z_k}}{1-\mod{z_k}} \leq \omega(z).$
\item[(b)] $\displaystyle\rho(0,z) \leq \frac{1}{2}\sum_{k = 1}^n\log\frac{1+\mod{z_k}}{1-\mod{z_k}}$.
\end{enumerate}
\end{lemma}

\begin{proof}  To prove (a), fix $z\in \D^n$ and $k = 1,\dots,n$, and for $w \in \D^n$, define $$h(w) = \frac{1}{2}\hbox{Log}\frac{\mod{z_k} + w_k\conj{z_k}}{\mod{z_k} - w_k\conj{z_k}},$$ where Log denotes the principal branch of the logarithm. Then $h\in H(\D^n)$, $h(0) = 0$, $\frac{\partial h}{\partial z_j}(w) = 0$ for $j \neq k$, and $\frac{\partial h}{\partial z_k}(w) = \frac{\conj{z_k}\mod{z_k}}{\mod{z_k}^2-w_k^2\conj{z_k}^2}.$  By Theorem~3.3 of \cite{CohenColonna:08}, $\beta_h(w) = \frac{(1-\mod{w_k}^2)\mod{z_k}^2}{\mod{\mod{z_k}^2 - w_k^2\conj{z_k}^2}} \leq 1.$  In particular, $$\frac{1}{2}\log\frac{1+\mod{z_k}}{1-\mod{z_k}} = \mod{h(z)} \leq \omega(z).$$

To prove (b), observe that for $z \in \D^n$ and $u\in\C^n$, $$H_z(u,\overline{u})=\sum_{k=1}^n\frac{|u_k|^2}{(1-|z_k|^2)^2}$$ (e.g. see \cite{Timoney:80-I}), and recall that if $\gamma=\gamma(t)$ ($0\leq t\leq 1$) is the geodesic from $w$ to $z$, then $$\rho(w,z)=\int_0^1 H_{\gamma(t)}(\gamma'(t),\overline{\gamma'(t)})^{1/2}\,dt.$$ Since the geodesic from 0 to $z\in\D^n$ is parametrized by $\gamma(t)=tz$, for $0\leq t\leq 1$, we obtain
$$\begin{aligned}
\rho(0,z) &= \int_0^1\left(\sum_{k=1}^n\frac{\mod{z_k}^2}{(1-\mod{z_k}^2t^2)^2}\right)^{1/2}\;dt \leq \int_0^1\sum_{k=1}^n \frac{\mod{z_k}}{1-\mod{z_k}^2t^2}\;dt\\ &= \frac{1}{2}\sum_{k=1}^n\log\frac{1+\mod{z_k}}{1-\mod{z_k}}.\;\qedhere
\end{aligned}$$
\end{proof}

\begin{theorem}\label{theorem: M_psi equiv} Let $\psi \in H(\D^n)$.  Then the following are equivalent:
\begin{enumerate}
\item[(a)] $M_\psi$ is bounded on $\Bloch(\D^n)$.
\item[(b)] $M_\psi$ is bounded on $\Bloch_{0^*}(\D^n)$.
\item[(c)] $\psi \in H^\infty(\D^n)$ and $\displaystyle\sup_{z \in \D^n}\sum_{k=1}^n \log\frac{1+\mod{z_k}}{1-\mod{z_k}}Q_\psi(z)$ is finite.
\end{enumerate}
\end{theorem}

\begin{proof} The equivalence $(a) \Longleftrightarrow (c)$ follows at once from Theorem \ref{theorem:bounded operator on bloch space}(a), Lemma \ref{lemma_polydisk}, and (\ref{omega_0_omega_rho}).

	To prove $(c) \Longrightarrow (b)$, suppose $\psi \in H^\infty(\D^n)$ and $\sup_{z \in \D^n}\sum_{k=1}^n \log\frac{1+\mod{z_k}}{1-\mod{z_k}}Q_\psi(z)$ is finite.  Since as $z \to \partial^*\D^n$, $\log\frac{1+\mod{z_k}}{1-\mod{z_k}} \to \infty$ for all $k \in \{1,\dots,n\}$, the finiteness of $\sup_{z \in \D^n} Q_\psi(z)\sum_{k=1}^n\log\frac{1+\mod{z_k}}{1-\mod{z_k}}$ implies that $Q_\psi(z) \to 0$ as $z \to \partial^*\D^n$, and thus $\psi \in \Bloch_{0^*}(D)$. Furthermore, by (\ref{omega_0_omega_rho}) and Lemma \ref{lemma_polydisk}(b), $\sigma_{0,\psi} < \infty$.  By Theorem \ref{theorem:bounded operator on bloch space}(b), $M_\psi$ is bounded on $\Bloch_{0^*}(\D^n)$.

	Conversely, assume $M_\psi$ is bounded on $\Bloch_{0^*}(\D^n)$. Then by Theorem \ref{theorem:bounded operator on bloch space}(b), $\psi \in H^\infty(\D^n) \cap \Bloch_{0^*}(\D^n)$ and $\sigma_{0,\psi} < \infty$. Fix $k \in \{1,\dots,n\}$ and $w \in \D^n$, and for $z \in \D^n$ define $$f_w(z) = \frac{1}{2}\mathrm{Log}\frac{1+\conj{w_k}z_k}{1-\conj{w_k}z_k}.$$  Observe that $f_w(0) = 0$ and, for $z\in\D^n$, $\frac{\partial f_w}{\partial z_j}(z) = 0$ for all $j \neq k$, and $$\frac{\partial f_w}{\partial z_k}(z) = \frac{\conj{w_k}}{1-\conj{w_k}^2z_k^2}.$$  From Theorem 3.3 of \cite{CohenColonna:08}, we have
\begin{eqnarray} Q_{f_w}(z)&=&\left\|\left((1-|z_1|^2)\mod{\frac{\partial f_w}{\partial z_1}(z)},\dots,(1-|z_n|^2)\mod{\frac{\partial f_w}{\partial z_n}(z)}\right)\right\|\nonumber\\ &=&\frac{(1-|z_k|^2)|w_k|}{|1-\conj{w_k}^2z_k^2|}\leq \frac{(1-\mod{z_k}^2)\mod{w_k}}{1-\mod{w_k}^2\mod{z_k}^2}\leq |w_k|.\nonumber\end{eqnarray}    Thus $\blochnorm{f_w} \leq \mod{w_k} < 1.$  Furthermore, $Q_{f_w}(z) \to 0$ as $z \to \partial^*\D^n$, and so $f_w \in \Bloch_{0^*}(\D^n)$.  Hence, for $z \in \D^n$, \begin{equation}\label{f_w < sigma_0}\mod{f_w(z)}Q_\psi(z) \leq \omega_{0,\psi}(z)Q_\psi(z) \leq \sigma_{0,\psi}.\end{equation}  Observe that
$$\mod{f_w(z)} \geq \frac{1}{2}\left(\log\mod{\frac{1+\conj{w_k}z_k}{1-\conj{w_k}z_k}} - \frac{\pi}{2}\right),$$ and so $$\frac{1}{2}\log\mod{\frac{1+\conj{w_k}z_k}{1-\conj{w_k}z_k}} \leq \mod{f_w(z)} + \frac{\pi}{4}.$$  Let $w_k = \mod{w_k}e^{i\theta_k}$ and choose $z_k$ so that $\arg(z_k) = \theta_k$.  From (\ref{f_w < sigma_0}), we obtain
$$\frac{1}{2}Q_\psi(z)\log\frac{1+\mod{w_k}\mod{z_k}}{1-\mod{w_k}\mod{z_k}} \leq \sigma_{0,\psi} + \frac{\pi}{4}\beta_\psi.$$  By letting $\mod{w_k} \to 1$, we obtain $$\sup_{z \in \D^n}Q_\psi(z)\log\frac{1+\mod{z_k}}{1-\mod{z_k}} \leq 2\sigma_{0,\psi} + \frac{\pi}{2}\beta_\psi.$$ We deduce that
$$\sup_{z\in \D^n}Q_\psi(z)\sum_{k=1}^n\log\frac{1+\mod{z_k}}{1-\mod{z_k}}\leq n\left(2\sigma_{0,\psi} + \frac{\pi}{2}\beta_\psi\right)<\infty,$$ proving that $(b)\Longrightarrow (c)$.
\end{proof}

\begin{corollary}\label{constantsymbol} Let $n \geq 2$ and $\psi \in H(\D^n)$.  Then $M_\psi$ is bounded on $\Bloch(\D^n)$ if and only if $\psi$ is a constant function.\end{corollary}

\begin{proof} It is immediate that $M_\psi$ is bounded on $\Bloch(\D^n)$ if $\psi$ is a constant function. Conversely, suppose $M_\psi$ is bounded on $\Bloch(\D^n)$. By Theorem \ref{theorem: M_psi equiv}, $\psi$ is a Bloch function and $$\sup_{z \in \D^n}Q_\psi(z)\sum_{k=1}^n\log\frac{1+\mod{z_k}}{1-\mod{z_k}} < \infty.$$
As $z \to \partial\D^n$, $\log\frac{1+\mod{z_k}}{1-\mod{z_k}} \to \infty$ for some $k \in \{1,\dots,n\}$. Thus, $$\lim_{z \to \partial\D^n} Q_\psi(z) = 0.$$  Since $n \geq 2$, the unit polydisk is not biholomorphically equivalent to the unit ball, and so by Proposition 4.1 of \cite{Timoney:80-II}, $\psi$ is a constant function.
\end{proof}

Finally, we provide sufficient conditions for the boundedness of $M_\psi$ on $\Bloch(D)$ to be equivalent to the boundedness on $\Bloch_{0^*}(D)$, where $D$ is any bounded homogeneous domain in $\C^n$.  The notation $A \asymp B$ means there exist constants $c_1,c_2 > 0$ such that $c_1A \leq B \leq c_2A$.

\begin{proposition} Let $D$ be a bounded homogeneous domain in $\C^n$ and $\psi \in H(D)$.  If $\sigma_\psi \asymp \sigma_{0,\psi}$ and $\displaystyle\lim_{z \to \distbd{D}} \omega_0(z) = \infty$, then $M_\psi$ is bounded on $\Bloch(D)$ if and only if it is bounded on $\Bloch_{0^*}(D)$.
\end{proposition}

\begin{proof} First, assume $M_\psi$ is bounded on $\Bloch(D)$. Then, $\psi \in H^\infty(D)$ and $\sigma_{0,\psi}$ is finite. From the hypothesis, it follows that $Q_\psi(z) \to 0$ as $z \to \distbd{D}$. Thus, $\psi\in \Bloch_{0^*}(D)$, proving that $M_\psi$ is bounded on $\Bloch_{0^*}(D)$. Conversely, if $M_\psi$ is bounded on $\Bloch_{0^*}(D)$, it follows immediately that $M_\psi$ is bounded on $\Bloch(D)$, since $\sigma_\psi \asymp \sigma_{0,\psi}$.
\end{proof}

%-------------------------------------------------------------------------------------------------
% Norm Estimates
%-------------------------------------------------------------------------------------------------
\section{Operator Norm Estimates}
In this section, we provide estimates on the norm of the bounded multiplication operators on the Bloch space and $*$-little Bloch space of a bounded homogeneous domain. % The upper estimates are described in terms of the quantities $\sigma_\psi$ and $\sigma_{0,\psi}$, respectively.
These estimates correspond to those established in \cite{AllenColonna:08} for the case of the unit disk.

\begin{theorem}\label{norm_estimate_bloch} Let $D$ be a bounded homogeneous domain in $\C^n$.
\begin{enumerate}
\item[\normalfont{(a)}] If $\psi \in H(D)$ induces a bounded multiplication operator on $\Bloch(D)$, then
$$\max\{\blochnorm{\psi}, \supnorm{\psi}\} \leq \norm{M_\psi} \leq \max\{\blochnorm{\psi}, \supnorm{\psi} + \sigma_\psi\}.$$

\item[\normalfont{(b)}] If $\psi \in H(D)$ induces a bounded multiplication operator on $\Bloch_{0^*}(D)$, then
$$\max\{\blochnorm{\psi}, \supnorm{\psi}\} \leq \norm{M_\psi} \leq \max\{\blochnorm{\psi}, \supnorm{\psi} + \sigma_{0,\psi}\}.$$
\end{enumerate}
\end{theorem}

\begin{proof} We will prove the norm estimates for $M_\psi$ bounded on $\Bloch(D)$, the argument for $\Bloch_{0^*}(D)$ is analogous.  By Lemma 11 of \cite{DurenRombergShields:69}, we have $\supnorm{\psi} \leq \norm{M_\psi}$. Furthermore, for $f$ identically 1, $\blochnorm{M_\psi f} = \blochnorm{\psi}$. Therefore $\norm{M_\psi} \geq \max\{\blochnorm{\psi},\supnorm{\psi}\}.$

Next, let $f \in \Bloch(D)$. and apply the identity $\mod{f(0)} = \blochnorm{f}-\beta_f$ to (\ref{equation:blochnorm_inequality_2}) to deduce
\begin{eqnarray} \blochnorm{M_\psi f} \leq \blochnorm{\psi}\blochnorm{f} + \left(\supnorm{\psi} + \sigma_\psi - \blochnorm{\psi}\right)\beta_f.\nonumber\end{eqnarray}
If $\supnorm{\psi} + \sigma_\psi \leq  \blochnorm{\psi}$, then $\blochnorm{M_\psi f} \leq \blochnorm{\psi}\blochnorm{f}$, while if $\supnorm{\psi} + \sigma_\psi \geq  \blochnorm{\psi}$, then
$\blochnorm{M_\psi f} \leq \left(\supnorm{\psi} + \sigma_\psi\right)\blochnorm{f},$ proving the upper estimate. \end{proof}

%-------------------------------------------------------------------------------------------------
% Spectra
%-------------------------------------------------------------------------------------------------
\section{Spectrum}
In this section, we determine the spectra of the bounded multiplication operators on the Bloch space and $*$-little Bloch space of a bounded homogeneous domain in $\C^n$, thereby extending the results obtained in \cite{AllenColonna:08} for the case of the unit disk.

Recall that the spectrum of an operator $T$ is defined as
$$\sigma(T) = \left\{\lambda \in \C : T - \lambda I \text{ is not invertible}\right\},$$ where $I$ is the identity operator.  If $\lambda \in \C$, then $M_\psi - \lambda I = M_{\psi - \lambda}.$  Thus, $\lambda \in \sigma(M_\psi)$ if and only if $M_{\psi - \lambda}$ is not invertible.

\begin{theorem}\label{theorem:spectrum on bloch space} Let $D$ be a bounded homogeneous domain in $\C^n$ and assume $\psi \in H(D)$ induces a bounded multiplication operator on $\Bloch(D)$ or $\Bloch_{0^*}(D)$.  Then $\sigma(M_\psi) = \overline{\psi(D)}.$\end{theorem}

\begin{proof}  We first prove the result for $M_\psi$ acting on $\Bloch(D)$.  Let $\lambda \in \psi(D)$.  Then there exists $z_0 \in D$ at which the function $(\psi(z) - \lambda)^{-1}$ is singular.  Thus $M_{\psi - \lambda}$ is not invertible and $\psi(D) \subseteq \sigma(M_\psi)$.  Since the spectrum is closed, $\overline{\psi(D)} \subseteq \sigma(M_\psi)$.

Suppose $\lambda \not\in \overline{\psi(D)}$.  Then $\psi(z) - \lambda$ is bounded away from zero; that is, there exists $\alpha > 0$ such that $\mod{\psi(z) - \lambda} \geq \alpha$ for all $z \in D$.  Thus the function $g$ defined by $g(z) = (\psi(z) - \lambda)^{-1}$ is bounded holomorphic on $D$.  By Theorem \ref{theorem:bounded operator on bloch space}, the multiplication operator induced by $g$ is bounded on $\Bloch(D)$ since $$\sigma_g = \sup_{z \in D}\;\omega(z)Q_g(z) \leq \sup_{z \in D}\;\frac{1}{\alpha^2}\omega(z)Q_\psi(z) = \frac{1}{\alpha^2}\sigma_\psi < \infty.$$  Thus $\lambda \not\in \sigma(M_\psi)$.

To prove the result for $M_\psi$ acting on $\Bloch_{0^*}(D)$, it suffices to show that the bounded holomorphic function $g$ defined above is in the $*$-little Bloch space for $\lambda \not\in \overline{\psi(D)}$.  Since $\psi \in \Bloch_{0^*}(D)$ and $Q_g(z) \leq \frac{1}{\alpha^2}Q_\psi(z)$, it follows that $Q_g(z) \to 0$ as $z \to \distbd{D}$, as desired.
\end{proof}

%-------------------------------------------------------------------------------------------------
% Compactness
%-------------------------------------------------------------------------------------------------
\section{Compactness} In this section, we characterize the compact multiplication operators on the Bloch space and the $*$-little Bloch space of a bounded homogeneous domain in $\C^n$.  The key to this characterization is the spectral theorem of compact operators due to Riesz (see Theorem 7.1 of \cite{Conway:90}).  If we apply Theorem \ref{theorem:spectrum on bloch space} to the spectral theorem of compact operators, we obtain the following lemma.

\begin{lemma}\label{corollary:compact operator} Let $D$ be a bounded homogeneous domain in $\C^n$ and $\psi \in H(D)$ such that $M_\psi$ is bounded on $\Bloch(D)$ or $\Bloch_{0^*}(D)$.  Then $\overline{\psi(D)}$ is non-empty and at most countably infinite.  Furthermore, if $\overline{\psi(D)}$ is a singleton, then $\overline{\psi(D)} = \{0\}$.\end{lemma}

\begin{theorem}  Let $D$ be a bounded homogeneous domain.  The only compact multiplication operator on $\Bloch(D)$ or $\Bloch_{0^*}(D)$ is the operator whose symbol is identically 0.
\end{theorem}

\begin{proof} Let $D$ be a bounded homogeneous domain and $\psi \in H(D)$.  Clearly if $\psi \equiv 0$, then $M_\psi$ is compact on $\Bloch(D)$ or $\Bloch_{0^*}(D)$.  Suppose $M_\psi$ is compact on $\Bloch(D)$ or $\Bloch_{0^*}(D)$.  By Lemma \ref{corollary:compact operator}, the range of $\psi$ is at most countable.  On the other hand, if the range of $\psi$ contains two distinct points, then it contains a continuum, and thus is uncountable.  Hence, the range of $\psi$ must be a single point and by Lemma \ref{corollary:compact operator}, $\psi$ must be identically 0.\end{proof}

%-------------------------------------------------------------------------------------------------
% Isometries
%-------------------------------------------------------------------------------------------------
\section{Isometries}
In this section, we characterize the isometric multiplication operators on the Bloch space of a large class of bounded symmetric domains in $\C^n$.

\begin{remark}\label{remark:isometry norm} First, we observe that if $M_\psi$ is an isometry on $\Bloch(D)$ for any bounded homogeneous domain $D$, then $\norm{M_\psi} = \blochnorm{\psi}$.  This follows immediately by applying $M_\psi$ to the constant function $1$.  In fact, if $\psi$ is any constant function of modulus one, then $M_\psi$ is an isometry on the Bloch space.  Of course, the same is true on the $*$-little Bloch space.\end{remark}

In \cite{CohenColonna:94}, Cohen and the second author defined the \emph{Bloch constant} of a bounded homogeneous domain $D$ in $\C^n$ as $$c_D = \sup\;\left\{\beta_f : f \in H(D), f(D) \subseteq \D\right\}.$$  When $D$ is a Cartan classical domain, the Bloch constants were computed to be $$c_D = \begin{cases}
\sqrt{2/(n+m)}, &\text{if } D \in R_I,\\
\sqrt{2/(n+1)}, &\text{if } D \in R_{II},\\
\sqrt{1/(n-1)}, &\text{if } D \in R_{III},\\
\sqrt{2/n}, &\text{if } D \in R_{IV}.
\end{cases}$$  In \cite{Zhang:97}, Zhang computed the Bloch constant of the two exceptional domains to be $1/\sqrt{6}$ and $1/3$.  By Theorem 3 of \cite{CohenColonna:94}, extended to include the exceptional domains, if $D = D_1\times\cdots\times D_k$ is a bounded symmetric domain in standard form, then \begin{equation}\label{max c_D}c_D = \max_{1\leq j\leq k} c_{D_j}.\end{equation}

\begin{lemma}\label{lemma:c_D standard form} Let $D = D_1\times\cdots\times D_k$ be a bounded symmetric domain in standard form.  Then $c_D \leq 1$, and $c_D = 1$ if and only if $D_j = \D$ for some $j \in \{1,\dots,k\}$.\end{lemma}

\begin{proof} First we consider the Bloch constant for each irreducible factor $D_j$ for $1 \leq j \leq k$.  Using the dimensional restrictions in Section \ref{section:preliminaries}, by inspection it is clear that $c_{D_j} \leq 1$ and $c_{D_j} = 1$ if and only if $D_j$ is in $R_I$ with $n=m=1$.  Thus, by (\ref{max c_D}), $c_D \leq 1$ and $c_D = 1$ if and only if there exists $j \in \{1,\dots,k\}$ such that $D_j$ is in $R_I$ with $n=m=1$.  If there exists such a $j$, then $D_j = \D$.\end{proof}

We denote by $\mathfrak{D}$ the set of bounded symmetric domains $D$ for which $c_D < 1$.  By Lemma \ref{lemma:c_D standard form}, these are precisely the bounded symmetric domains which do not have $\D$ as a factor when written in standard form.  The next result follows immediately from Remark \ref{remark:isometry norm} using induction on $k$.

\begin{lemma} Let $D$ be a bounded homogeneous domain in $\C^n$ and $\psi \in H(D)$.  If $M_\psi$ is an isometry on $\Bloch(D)$ $($or $\Bloch_{0^*}(D))$, then $M_{\psi^k}$ is an isometry on $\Bloch(D)$ $(\text{respectively, } \Bloch_{0^*}(D))$ for all $k \in \N$.  In particular, $\blochnorm{\psi^k} = 1$ for all $k \in \N$.
\end{lemma}

\begin{lemma}\label{lemma:seminorm_inequality} Let $D$ be a bounded symmetric domain and $\psi \in H(D)$.  If $M_\psi$ is an isometry on $\Bloch(D)$ or $\Bloch_{0^*}(D)$, then $\beta_{\psi^k} \leq c_D$ for all $k \in \N$.  In particular, if $D \in \mathfrak{D}$, then $\beta_{\psi^k} < 1$ for all $k \in \N$.
\end{lemma}

\begin{proof} By Lemma 11 of \cite{DurenRombergShields:69}, $\supnorm{\psi^k} \leq \norm{M_{\psi^k}} = 1$.  Thus $\psi^k$ is either a constant function of modulus one or a bounded holomorphic function mapping into $\D$.  If $\psi^k$ is a constant function of modulus one, then $\beta_{\psi^k} = 0$.  If $\psi^k$ is bounded holomorphic mapping into $\D$, then the definition of the Bloch constant of $D$ gives $\beta_{\psi^k} \leq c_D$ for all $k \geq 1$.  Thus in both cases, $\beta_{\psi^k} \leq c_D$ for all $k \in \N$.  If $D \in \mathfrak{D}$, the conclusion follows from Lemma \ref{lemma:c_D standard form}.
\end{proof}

\begin{theorem}\label{theorem:isometry_characterization} Let $D \in \mathfrak{D}$ and $\psi \in H(D)$.  Then $M_\psi$ is an isometry on $\Bloch(D)$ or $\Bloch_{0^*}(D)$ if and only if $\psi$ is a constant function of modulus one.\end{theorem}

\begin{proof} We will prove this statement for $\Bloch(D)$ (the proof follows the same argument for $\Bloch_{0^*}(D)$).  Since the multiplication operators induced by unimodular constants are isometries, we only need to prove the converse.  Suppose $M_\psi$ is an isometry on $\Bloch(D)$ and $\psi$ is not a constant function of modulus one.  Then $\psi(0) = a$, for some $\mod{a} < 1$.  Since $\blochnorm{\psi^k} = 1$, by Lemma \ref{lemma:seminorm_inequality} we have $$\mod{a}^k = 1-\beta_{\psi^k} \geq 1-c_D,$$ so that $\mod{a}^k$ is bounded away from 0, contradicting the fact that $\mod{a}^k \to 0$ as $k \to \infty$.  Thus, if $M_\psi$ is an isometry, then $\psi$ must be a constant function of modulus one.
\end{proof}

%-------------------------------------------------------------------------------------------------
% Open Questions
%-------------------------------------------------------------------------------------------------
\section{Open Questions}
Let $D$ be a bounded homogeneous domain.

\begin{enumerate}
\item If $\psi \in H(D)$ induces a bounded multiplication operator on $\Bloch(D)$, are the upper bounds of Theorem \ref{norm_estimate_bloch} sharp?

\item If $D$ is not conformally equivalent to the unit ball and $\rho$ is the Bergman distance on $D$, is $\omega(z) = \rho(0,z)$ for each $z \in D$?

\item If $D$ is a bounded homogeneous domain other than the ball:
\begin{enumerate}
\item Is $\omega(z) = \omega_0(z)$ for each $z \in D$?
\item Is $\sigma_\psi \asymp \sigma_{0,\psi}$ for some $\psi \in H(D)$?  If so, which functions $\psi$ satisfy this property?
\item Is $\displaystyle\lim_{z \to \distbd{D}} \omega_0(z) = \infty$?
\end{enumerate}
\end{enumerate}

\begin{remark} Theorem \ref{theorem:isometry_characterization} pertains to bounded symmetric domains which do not have $\D$ as a factor, and so it does not apply to the case of the Bloch space on the unit disk.  However, in \cite{AllenColonna:08}, the authors showed that even in this case the isometric multiplication operators on the Bloch space of $\D$ are precisely those operators whose symbol is a constant function of modulus one.  This result was known to N. Zorboska, R. Zhao and Z. Cuckovic. Furthermore, by Corollary~\ref{constantsymbol}, even in the case of the unit polydisk, the only isometric multiplication operators are those whose symbol is a constant of modulus one. Thus we end this paper with the following conjecture.
\end{remark}

\begin{conjecture}Let $D$ be a bounded homogeneous domain.  Then $M_\psi$ is an isometry on $\Bloch(D)$ or $\Bloch_{0^*}(D)$ if and only if $\psi$ is a constant function of modulus one.\end{conjecture}

%-------------------------------------------------------------------------------------------------
% Bibliography
%-------------------------------------------------------------------------------------------------
\bibliographystyle{amsplain}
\bibliography{references.bib}

\providecommand{\bysame}{\leavevmode\hbox to3em{\hrulefill}\thinspace}
\providecommand{\MR}{\relax\ifhmode\unskip\space\fi MR }
\providecommand{\MRhref}[2]{%
  \href{http://www.ams.org/mathscinet-getitem?mr=#1}{#2}
}
\providecommand{\href}[2]{#2}
\begin{thebibliography}{10}

\bibitem{AllenColonna:08}
Robert~F. Allen and Flavia Colonna, \emph{Isometries and spectra of
  multiplication operators on the {B}loch space}, Bull. Aust. Math. Soc.
  \textbf{79} (2009), no.~1, 147--160. \MR{2486890}

\bibitem{AllenColonna:07}
\bysame, \emph{On the isometric composition operators on the {B}loch space in
  {$\Bbb C^n$}}, J. Math. Anal. Appl. \textbf{355} (2009), no.~2, 675--688.
  \MR{2521743}

\bibitem{Arazy:82}
Jonathan Arazy, \emph{Multipliers of bloch functions}, University of Haifa
  Mathem. Public. Series \textbf{54} (1982).

\bibitem{BrownShields:91}
Leon Brown and Allen~L. Shields, \emph{Multipliers and cyclic vectors in the
  {B}loch space}, Michigan Math. J. \textbf{38} (1991), no.~1, 141--146.
  \MR{1091517}

\bibitem{Cartan:35}
Elie Cartan, \emph{Sur les domaines born\'{e}s homog\`enes de l'espace den
  variables complexes}, Abh. Math. Sem. Univ. Hamburg \textbf{11} (1935),
  no.~1, 116--162. \MR{3069649}

\bibitem{CohenColonna:08}
Joel Cohen and Flavia Colonna, \emph{Isometric composition operators on the
  {B}loch space in the polydisk}, Banach spaces of analytic functions, Contemp.
  Math., vol. 454, Amer. Math. Soc., Providence, RI, 2008, pp.~9--21.
  \MR{2408231}

\bibitem{CohenColonna:94}
Joel~M. Cohen and Flavia Colonna, \emph{Bounded holomorphic functions on
  bounded symmetric domains}, Trans. Amer. Math. Soc. \textbf{343} (1994),
  no.~1, 135--156. \MR{1176085}

\bibitem{Conway:90}
John~B. Conway, \emph{A course in functional analysis}, second ed., Graduate
  Texts in Mathematics, vol.~96, Springer-Verlag, New York, 1990. \MR{1070713}

\bibitem{CowenMacCluer:95}
Carl~C. Cowen and Barbara~D. MacCluer, \emph{Composition operators on spaces of
  analytic functions}, Studies in Advanced Mathematics, CRC Press, Boca Raton,
  FL, 1995. \MR{1397026}

\bibitem{Drucker:78}
Daniel Drucker, \emph{Exceptional {L}ie algebras and the structure of
  {H}ermitian symmetric spaces}, Mem. Amer. Math. Soc. \textbf{16} (1978),
  no.~208, iv+207. \MR{499340}

\bibitem{DurenRombergShields:69}
Peter~L. Duren, Bernhard~W. Romberg, and Allen~L. Shields, \emph{Linear
  functionals on {$H^{p}$} spaces with {$0<p<1$}}, J. Reine Angew. Math.
  \textbf{238} (1969), 32--60. \MR{259579}

\bibitem{Hahn:75}
Kyong~T. Hahn, \emph{Holomorphic mappings of the hyperbolic space into the
  complex {E}uclidean space and the {B}loch theorem}, Canadian J. Math.
  \textbf{27} (1975), 446--458. \MR{466641}

\bibitem{Kobayashi:05}
Shoshichi Kobayashi, \emph{Hyperbolic manifolds and holomorphic mappings},
  second ed., World Scientific Publishing Co. Pte. Ltd., Hackensack, NJ, 2005,
  An introduction. \MR{2194466}

\bibitem{KrantzMa:88}
Steven~G. Krantz and Daowei Ma, \emph{On isometric isomorphisms of the {B}loch
  space on the unit ball of {${\bf C}^n$}}, Michigan Math. J. \textbf{36}
  (1989), no.~2, 173--180. \MR{1000521}

\bibitem{OhnoZhao:01}
Sh\^{u}ichi Ohno and Ruhan Zhao, \emph{Weighted composition operators on the
  {B}loch space}, Bull. Austral. Math. Soc. \textbf{63} (2001), no.~2,
  177--185. \MR{1823706}

\bibitem{Pjat:59}
Ilya~I. Pjatecki\u{\i}-\v{S}apiro, \emph{On a problem proposed by {E}.
  {C}artan}, Dokl. Akad. Nauk SSSR \textbf{124} (1959), 272--273. \MR{0101922}

\bibitem{Shapiro:93}
Joel~H. Shapiro, \emph{Composition operators and classical function theory},
  Universitext: Tracts in Mathematics, Springer-Verlag, New York, 1993.
  \MR{1237406}

\bibitem{Timoney:80-I}
Richard~M. Timoney, \emph{Bloch functions in several complex variables. {I}},
  Bull. London Math. Soc. \textbf{12} (1980), no.~4, 241--267. \MR{576974}

\bibitem{Timoney:80-II}
\bysame, \emph{Bloch functions in several complex variables. {II}}, J. Reine
  Angew. Math. \textbf{319} (1980), 1--22. \MR{586111}

\bibitem{Zhang:97}
Genkai Zhang, \emph{Bloch constants of bounded symmetric domains}, Trans. Amer.
  Math. Soc. \textbf{349} (1997), no.~7, 2941--2949. \MR{1329540}

\bibitem{ZhouChen-I:05}
Zehua Zhou and Renyu Chen, \emph{Weighted composition operators between
  different bloch-type spaces in polydisk}, 2005, {arXiv}:math/0503622
  [math.CV].

\bibitem{ZhouChen:05}
\bysame, \emph{Weighted composition operators from $f(p,q,s)$ to bloch type
  spaces on the unit ball},  (2005), {arXiv}:math/0503614 [math.CV].

\bibitem{Zhu:04}
Kehe Zhu, \emph{Spaces of holomorphic functions in the unit ball}, Graduate
  Texts in Mathematics, vol. 226, Springer-Verlag, New York, 2005. \MR{2115155}

\end{thebibliography}
\end{document}